\documentclass[reqno]{amsart}
\usepackage{amsfonts}
\usepackage{graphicx}
\usepackage{amscd}
\usepackage{amsmath}
\usepackage{amssymb}
\usepackage{latexsym}
\usepackage{hyperref}
\usepackage[all]{xy}
\usepackage[latin1]{inputenc}
\theoremstyle{plain}

\newtheorem{corollary}{\bf Corollary}

\newtheorem{example}{\bf Example}
\newtheorem{lemma}{\bf Lemma}

\newtheorem{proposition}{\bf Proposition}
\newtheorem{remark}{Remark}

\newtheorem{theorem}{\bf Theorem}
\numberwithin{equation}{section}

\newcommand\dv{\mathrm{div}}

\begin{document}

\title[On the construction of gradient Ricci soliton warped product]{On the construction of gradient Ricci soliton warped product}

\author{F.E.S. Feitosa, A.A. Freitas Filho}
\address{(F.E.S. Feitosa) Departamento de Matem\'atica-ICE-UFAM, 69080-900, Manaus-AM-BR}
\email{feitosaufam@gmail.com}

\address{(A.A. Freitas Filho) Departamento de Matem\'atica-ICE-UFAM, 69080-900, Manaus-AM-BR}
\email{freitas.airtonfilho@hotmail.com}

\author{J.N.V. Gomes}
\address{(J.N.V. Gomes) Departamento de Matem\'atica-ICE-UFAM, 69080-900, Manaus-AM-BR}
\email{jnvgomes@pq.cnpq.br}
\urladdr{http://www.ufam.edu.br}

\keywords{Ricci soliton; Warped product; Scalar curvature; Rigidity results}

\subjclass[2010]{Primary 53C24, 53C25; Secondary 53C15, 53C21}

\begin{abstract}
In this paper we show that an expanding or steady gradient Ricci soliton warped product $B^n\times_f F^m$, $m>1$, whose warping function $f$ reaches both maximum and minimum must be a Riemannian product. Moreover, we present a necessary and sufficient condition for constructing a gradient Ricci soliton warped product. As an application, we present a class of expanding Ricci soliton warped product having as a fiber an Einstein manifold with non-positive scalar curvature. We also discuss some obstructions to this construction, especially in the case when the base of the warped product is compact.
\end{abstract}
\maketitle

\section{Introduction}

The study of warped products have been of great interest throughout the recent years. This concept was first introduced by Bishop and O'Neill as they succeeded  to give examples of complete Riemannian manifolds with negative sectional curvature \cite{BO}. Given two Riemannian manifolds $(B,g_{B})$ and $(F,g_{F})$ as well as a positive smooth function $f$ on $B$, we define on the product manifold $B\times F$ the metric
\begin{equation}\label{WPm}
g=\pi^{*}g_{B}+(f\circ \pi)^{2}\sigma^{*}g_{ F},
\end{equation}
where $\pi$ and $\sigma$ are the natural projections on $B$ and $F$, respectively. Under these conditions the product manifold is said to be the {\it warped product} of $B$ and $F$; it is denoted $M=B\times_{f} F$ and the function $f$ is called the \textit{warping function}. Notice that when $f$ is constant  $M$ is just the usual  Riemannian product. Albeit the class of warped products with non-constant warping functions provides a rich class of examples in Riemannian geometry, it was shown by Kim-Kim \cite{kim} that there does not exist a compact Einstein warped product  with non-constant warping function if the scalar curvature is non-positive. Moreover, they observed that a necessary condition for a warped product be an Einstein manifold is its base be a quasi-Einstein metric, i.e., a Riemannian manifold whose modified Bakry-Emery Ricci tensor is a constant multiple of the metric tensor. One should point out that some examples of expanding quasi-Einstein manifolds having as a fiber an arbitrary Einstein manifold as well as steady quasi-Einstein manifolds with fiber of non-negative scalar curvature were constructed in \cite {Besse}. More recently, Barros-Batista-Ribeiro \cite{BBR} provided some volume estimates for Einstein warped products similar to a classical result due to Calabi \cite{Calabi} and Yau \cite{Yau} for complete Riemannian manifolds with non-negative Ricci curvature. For this, they made use of the approach of quasi-Einstein manifolds. In particular, they also presented an obstruction for the existence of such a class of manifolds. We would like to mention here the work of He-Petersen-Wylie \cite{hepeterwylie} concerning warped product Einstein manifolds. Being an extension of the work of Case-Shu-Wei \cite{case1} and some earlier work of Kim-Kim \cite{kim}, the upshot of \cite{hepeterwylie} is that the base may have non-empty boundary.

A natural generalisation of the Einstein manifolds are the Ricci Solitons. This concept was introduced by Hamilton \cite{hamilton1} in early $80's$. We recall that a Ricci soliton is a complete Riemannian manifold $M$ endowed with a metric $g$, a vector field $X\in\mathfrak{X}(M)$ and a constant $\lambda$ satisfying the equation
\begin{eqnarray} \label{Ric1}
Ric+\frac{1}{2}\mathcal{L}_{X}g=\lambda g.
\end{eqnarray}
We shall refer to this equation as the fundamental equation. A Ricci soliton is called {\it expanding}, {\it steady} or {\it shrinking} if $\lambda<0$, $\lambda=0$ or $\lambda>0$, respectively. When $X=\nabla\psi$ for some smooth function $\psi$ on $M$, we write $(M,g,\nabla\psi,\lambda)$ for the gradient Ricci soliton with potencial function $\psi$. In this case, the fundamental equation can be rewritten as
\begin{eqnarray} \label{Ric2}
Ric+\nabla^{2}\psi=\lambda g,
\end{eqnarray}
where $\nabla^{2}\psi$ denotes the Hessian of $\psi$. For more details see \cite{cao,hamilton1}. It has been known since the early 90's that a compact gradient steady or expanding Ricci soliton is necessarily an Einstein manifold  \cite{hamilton2,Ivey}. In \cite{peterneW}, Petersen and Wylie used a theorem due to Brinkmann \cite{Brinkmann} to show that any surface gradient Ricci soliton is a warped product. It is also known that Robert Bryant (see \cite{Bryant,ChowEtal}) constructed a steady Ricci soliton as the warped product $(0,+\infty) \times_f \Bbb{S}^{m}$, $m>1$, with a radial warping function $f$. Since this latter function is not limited we reach the following natural question:
{\it Under which conditions a warped product with a limited warping function is a Ricci soliton?} Our first theorem gives a partial answer to this question.

\begin{theorem}\label{mainthm}
Let $M=B^n\times_f F^m$ be a warped product and $\varphi$ a smooth function on $B$ so that $(M,g,\nabla\tilde\varphi,\lambda)$ be an expanding or steady gradient Ricci soliton. Assume that its fiber $F^m$ is of dimension at least two and that its warping function $f$ reaches both maximum and minimum. Then $M$ must be a Riemannian product.
\end{theorem}

This latter theorem is motivated by the ideas of \cite{kim} which concern compact Einstein warped product spaces with non-positive scalar curvature. We point out that Theorem \ref{mainthm} is a natural generalisation of the Einstein case to the Ricci soliton case without the compactness condition on the product that was taken in \cite{kim}. Incidentally, an interesting fact emerges when we study Ricci solitons that are realised as a warped product. Indeed, their bases satisfy the equation \eqref{EQMthmIntr} below. This is a generalisation of the Einstein metrics, which contains quasi-Einstein metrics (see p. 6).

The next result establishes a compactness criterion of shrinking gradient Ricci soliton warped product under the condition that the base is compact.

\begin{theorem}\label{thmCpt}
Let $M=B^n\times_f F^m$ be a warped product and $\varphi$ a smooth function on $B$ so that $(M,g,\nabla\tilde\varphi,\lambda)$ be a shrinking gradient Ricci soliton with compact base and fiber with dimension at least two. Then $M$ must be a compact manifold.
\end{theorem}

It should be emphasized that there are several further interesting obstructions to the existence of Einstein metrics. For more details about this subject see \cite{BBR,BRS,Besse,case2,kim}. According to this fact, it is natural to investigate the obstruction results for Ricci solitons that are realised as a warped product. Beyond  Theorems \ref{mainthm} and \ref{thmCpt}, we will prove some more in the course of this paper, as it can be seen in the remarks and corollaries.

We next obtain a necessary and sufficient condition for constructing a gradient Ricci soliton warped product. For this reason, we consider a Riemannian manifold $(B^{n},g_B)$  with two smooth functions $f>0$ and $\varphi$ satisfying
\begin{equation}\label{EQMthmIntr}
Ric+\nabla^{2}\varphi=\lambda g_B+\frac{m}{f}\nabla^{2}f
\end{equation}
and
\begin{equation}\label{EQMthmIntr2}
2\lambda\varphi-|\nabla\varphi|^{2}+\Delta\varphi+\frac{m}{f}\nabla\varphi(f)=c
\end{equation}
for some constants $m,c,\lambda\in\Bbb{R}$, with $m\neq0$. We will prove that $f$ and $\varphi$ satisfy
\begin{equation}\label{CMthmIntr}
\lambda f^{2}+f\Delta f + (m-1)|\nabla f|^{2} -f\nabla\varphi(f)=\mu
\end{equation}
for a constant $\mu\in\Bbb{R}$, cf. Proposition \ref{PP5}.

By taking $m$ to be an integer at least $2$ and using the Bishop and O'Neill formulas (cf. Lemmas \ref{oneil1} and \ref{oneil2}), we construct a gradient Ricci soliton warped product as follows.
\begin{theorem} \label{CRSPW}
Let $(B^{n},g_B)$ be a complete Riemannian manifold with two smooth functions $f>0$ and $\varphi$ satisfying \eqref{EQMthmIntr} and \eqref{EQMthmIntr2}. Take the constant $\mu$ satisfying \eqref{CMthmIntr} and a complete Riemannian manifold  $(F^m,g_F)$ with Ricci tensor $^F\!Ric=\mu g_F$ and $m>1$. Then $(B^n\times_fF^m,g,\nabla\tilde\varphi,\lambda)$ is a gradient Ricci soliton warped product.
\end{theorem}

At this juncture we should stress the fact that the requirement that the metric $g_F$ is Einstein is indispensable, cf. Proposition \ref{PP2}. Furthermore, $\mu$ is necessarily constant when the dimension of the fiber is at least 2, cf. Proposition \ref{PP5}.

As an application we will construct a class of expanding Ricci soliton warped product having as a fiber an Einstein manifold with non-positive scalar curvature, cf. Corollary \ref{CorCRSWPgeneral}.

Recently, our Theorem \ref{CRSPW} along with other excellent results have been proven in the case of steady gradient Ricci solitons warped product when the base is conformal to an $n(\geq3)$-dimensional pseudo Euclidean space invariant under the action of an $(n-1)$-dimensional translation group (see \cite{romildo}). As a consequence of the ODE's theory we observe that the technique of \cite{romildo} only applies to the construction of a steady Ricci soliton.

Other relevant works to be cited are Ivey \cite{Ivey2} and Dancer-Wang \cite{DW}. The outcome of these papers is the construction of noncompact gradient steady solitons, which was achieved by the use of double and multiple warped products. Notice that this construction is a generalisation of the construction of the Bryant's soliton. Also, Gastel-Kornz \cite{GK} constructed a two-parameter family (doubly warped product metrics) of gradient expanding solitons on $\Bbb{R}^{n}\times F^m$, where $F^m$ $(m\geq2)$ is an Einstein manifold with positive scalar curvature.

\section{Preliminaries}

In this section we shall follow the notation and terminology of Bishop and O'Neill \cite{BO}. Our immediate goal is to relate the calculus of $M=B\times F$ to that of its factors. The crucial notion for this is that of a \textit{lifting}. We consider the lift $\tilde{f}=f\circ\pi$  of $f$ to $M=B\times F$ of the a smooth real-valued function $f$ on $B$ and the lift of $X\in\mathfrak{X}(B)$ to $M$ is the vector field $\tilde{X}\in \mathfrak{X}(M)$ whose value at each $(p,q)$ is the unique vector $\tilde{X}\in T_{(p,q)}M$ such that $d\pi(\tilde{X})=X$. Thus the lift of $X$ to $M$ is the \textit{unique} element of $\mathfrak{X}(M)$ that is $\pi$-related to $X$ and $\sigma$-related to the zero vector field on $F$. The set of all such horizontal lifts $\tilde{X}$ is denoted by $\mathfrak{L}(B)$. Functions and vector fields on $F$ are lifted to $M$ in the same way using the projection $\sigma$. The set of all such vertical lifts $\tilde{V}$ is denoted by $\mathfrak{L}(F)$. From now on, if $X\in\mathfrak{X}(B)$, when there is no danger of confusion, we will use the same notation for its horizontal lift $X\in\mathfrak{L}(B)$; similarly for the vertical lift $V\in\mathfrak{L}(F)$ of $V\in\mathfrak{X}(F)$.

Recall that the warped product $M=B^n\times_{f}F^m$ of two Riemannian manifolds is simply their Riemannian product endowed with the metric \eqref{WPm}.
The manifold $B$ is called the \textit{base} of $M$ and  $F$ the \textit{fiber}. Tangent vectors to the leaves are \textit{horizontal} and tangent vectors to  the fibers are \textit{vertical}. We denote by $\mathcal{H}$ the orthogonal projection of $T_{(p,q)}M$ onto its horizontal subspace $T_{(p,q)}(B\times q)$, and by $\mathcal{V}$ the projection onto the vertical subspace $T_{(p,q)}(p\times F)$. It is well known that the gradient of the lift $h\circ\pi$ of a smooth function $h$ on $B$ to $M$ is the lift  of the gradient of $h$. Thus there should be no confusion if we simplify the notation by writing $\tilde h$ for $h\circ\pi$, so that the gradient, the Hessian and the Laplacian of $\tilde h$ calculated in the metric of $M$ are denoted respectively by $\nabla\tilde h$, $\nabla^2\tilde h$ and $\Delta\tilde h$, where  $\Delta=tr(\nabla^{2})$. We will denote by $D$, $\nabla$ and $^{F}\nabla$ the Levi-Civita connections of the $M$, $B$ and $F$, respectively. The following result is crucial for us.
\begin{lemma}[\cite{BO}]\label{oneil1}
On $M=B^n\times_{f}F^m$, if $Y,Z\in\mathfrak{L}(B)$ and $V,W\in\mathfrak{L}(F)$, then
\begin{itemize}
\item [(i)] $D_{Y}Z$ is the lift of $\nabla_{Y}Z$ on $B$,
\item [(ii)] $D_{Y}V=D_{V}Y=\frac{Y(f)}{f}V$,
\item [(iii)]$\mathcal{H}(D_{V}W)=-\frac{g(V,W)}{f}\nabla f$,
\item [(iv)] $\mathcal{V}(D_{V}W)\in\mathfrak{L}(F)$ is the lift of $^{F}\nabla_{V}W$ on $F$.
\end{itemize}
In particular,
\begin{equation}\label{WP1}
\Delta\tilde{h}=\Delta h+\frac{m}{f}\nabla h(f),
\end{equation}
for every smooth function $h$ on $B$.
\end{lemma}

In what follows we shall write $Ric$ for the Ricci tensor of the warped product, $^{B}\!Ric$ for the lift of the Ricci tensor of $B$ and $^{F}\!Ric$ for the lift of the Ricci tensor of $F$. Moreover, we denote by $H^{h}$ the lift  of the Hessian $\nabla^2h$ of a smooth function $h$ on $B$ to $M$. Observe that for all $Y,Z\in\mathfrak{L}(B)$ we have  $\nabla^2\tilde h (Y,Z)=H^h(Y,Z)$.

\begin{lemma}[\cite{BO}] \label{oneil2}\label{oneil2}
On $M=B^n\times_{f}F^m$, $m>1$, let $Y,Z\in\mathfrak{L}(B)$ and  $V,W\in\mathfrak{L}(F)$. Then
\begin{itemize}
\item [(i)] $Ric(Y,Z)=\  ^{B}\!Ric(Y,Z)-\frac{m}{f}H^{f}(Y,Z)$,
\item [(ii)] $Ric(Y,V)=0$,
\item [(iii)] $Ric(V,W)=\ ^{F}\!Ric(V,W)- \big(\frac{\Delta f}{f}+\frac{|\nabla f|^2}{f^2}(m-1)\big)g(V,W)$.
\end{itemize}
\end{lemma}

Now, given a gradient Ricci soliton $(M^{k},g,\nabla\psi,\lambda)$, we take the trace of the equation \eqref{Ric2} to obtain
\begin{equation*}
R+\Delta\psi=k\lambda.
\end{equation*}
Moreover, Hamilton \cite{hamilton2} proved that
\begin{equation}\label{Ham-C}
2\lambda\psi - |\nabla\psi|^{2}+ \Delta\psi=c,
\end{equation}
for some constant $c$. In this direction we deduce similar equation to \eqref{Ham-C} for the base of the a gradient Ricci soliton warped product, cf. equation \eqref{EqNecessary} below. This is the first result of the next section.

\section{Existence conditions for the Ricci soliton warped product}

Now we study a Riemannian manifold $(B^n,g_B)$ as possible base of a gradient Ricci soliton warped product $(M=B^n\times_{f}F^m,g,\nabla\psi,\lambda)$. This way, it is natural to suppose that the potential function $\psi$ is the lifting of a smooth function $\varphi$ defined in $B^n$, i.e., the base will carry crucial information of $M$, like on an Einstein warped product case as well as all the known examples in the literature. With this considerations in mind, we establish restrictions on the functions that parametrise a gradient Ricci soliton warped product. The first one is the Hamilton's equation \eqref{Ham-C} for $B^n$.

\begin{proposition} \label{Lem2}
Let $M=B^n\times_f F^m$ be a warped product and $\varphi$ a smooth function on $B$ so that $(M,g,\nabla\tilde\varphi,\lambda)$ is a gradient Ricci soliton. Then we have
\begin{equation}\label{EqNecessary}
2\lambda\varphi-|\nabla\varphi|^{2}+\Delta\varphi+\frac{m}{f}\nabla\varphi(f)= c
\end{equation}
for some constant $c$.
\end{proposition}
\begin{proof}
From \eqref{Ham-C} we have
\begin{equation} \label{L2-1}
2\lambda\tilde{\varphi}-|\nabla\tilde{\varphi}|^{2}+\Delta\tilde{\varphi}=c
\end{equation}
for some constant $c$. On the other hand,
\begin{eqnarray} \label{L2-2}
\nabla\tilde{\varphi}=\widetilde{\nabla\varphi} \ \ \mbox{and} \ \ \Delta\tilde{\varphi}=\Delta\varphi+\frac{m}{f}\nabla\varphi(f)
\end{eqnarray}
and substituting \eqref{L2-2} in \eqref{L2-1} we obtain immediately equation \eqref{EqNecessary}.
\end{proof}

\begin{proposition} \label{PP2}
Let $M=B^n\times_f F^m$ be a warped product and $\varphi$ a smooth function on $B$ so that $(M,g,\nabla\tilde\varphi,\lambda)$ is a gradient Ricci soliton, with $m>1$. Then we have
\begin{equation}\label{EqHteq}
^{B}\!Ric+H^{\varphi}=\lambda g_{B}+\frac{m}{f}H^{f}
\end{equation}
and $^{F}\!Ric=\mu g_{F}$ with $\mu$ satisfying
\begin{equation}\label{EqHteqAux}
\mu=\lambda f^{2}+f\Delta f+(m-1)|\nabla f|^{2}-f\nabla\varphi(f).
\end{equation}
\end{proposition}
\begin{proof}
First we observe that for all $Y,Z\in\mathfrak{L}(B)$ we have
\begin{equation*}
Ric(Y,Z)=\!^{B}\!Ric(Y,Z)-\frac{m}{f}H^{f}(Y,Z).
\end{equation*}
Using the fundamental equation \eqref{Ric2} and the fact that $\nabla^{2}\tilde{\varphi}(Y,Z)=H^{\varphi}(Y,Z)$ we deduce that
\begin{equation*}
^{B}\!Ric(Y,Z)=\lambda g_{B}(Y,Z)-H^{\varphi}(Y,Z)+\frac{m}{f}H^{f}(Y,Z).
\end{equation*}
This proves the first assertion of the proposition. Similarly, for $V,W\in\mathfrak{L}(F)$ we have
\begin{eqnarray*}
^{F}\!Ric(V,W)&=&\lambda g(V,W)-\nabla^{2}\tilde{\varphi}(V,W)+\Big(\frac{\Delta f}{f}+(m-1)\frac{|\nabla f|^{2}}{f^{2}}\Big)g(V,W)\\
&=&\lambda f^{2}g_{F}(V,W)-\nabla^{2}\tilde{\varphi}(V,W)+f\Big(\Delta f+(m-1)\frac{|\nabla f|^{2}}{f}\Big)g_{F}(V,W).
\end{eqnarray*}
Since $\nabla\tilde\varphi\in\mathfrak{L}(B)$ we get
\begin{equation}\label{eqAux1Thm2}
\nabla^{2}\tilde{\varphi}(V,W) = g(D_{V}\nabla\tilde\varphi,W)= g\Big(\frac{\nabla\widetilde\varphi(f)}{f}V,W\Big) = f\nabla\varphi(f)g_{F}(V,W).
\end{equation}
Thus,
\begin{equation*}
^{F}\!Ric(V,W)=\big(\lambda f^{2}+f\Delta f+(m-1)|\nabla f|^{2}-f\nabla\varphi(f)\big)g_{F}(V,W).
\end{equation*}
Hence we complete the proof of the proposition.
\end{proof}

The main purpose of this article is to consider Ricci solitons that are warped products, so the previous proposition justifies why we considered equation \eqref{EQMthmIntr} in the introduction. Motivated by the work of Maschler \cite{maschler},  we will refer to this equation as the Ricci-Hessian type equation. In \cite{case1}, Case-Shu-Wei introduced the concept of quasi-Einstein metrics which originated from the usual study of Einstein manifolds that are realised as warped products, cf. Besse \cite{Besse}. Thus, it was naturally expected that the class of Ricci-Hessian type equation contains the class of quasi-Einstein metrics. In order to prove this assertion, let $(B^n,g,h,\lambda)$ be a quasi-Einstein metric, that is
\begin{equation*}
Ric+\nabla^2h-\frac{1}{m}dh\otimes dh=\lambda g
\end{equation*}
for some $\lambda\in\Bbb{R}$ and $0<m\leq\infty$. Taking $m=4r<\infty$, $\varphi=\frac{h}{2}$ and $f=e^{-\frac{\varphi}{r}}$, we get
\begin{equation*}
\frac{r}{f}\nabla^2f=-\nabla^2\varphi + \frac{1}{r}d\varphi\otimes d\varphi.
\end{equation*}
Thus, by straightforward computation we have that $(B^n,g,\varphi,f)$ satisfies the Ricci-Hessian type equation, namely
\begin{equation}\label{EqAuxEx1}
Ric+\nabla^2\varphi =\lambda g+\frac{r}{f}\nabla^2f.
\end{equation}
For $m=\infty$, we must consider $\varphi=h$ and $f=constant$.

Now let us suppose that $(B^n,g,\varphi,f,\lambda)$ satisfies an equation of the type \eqref{EqAuxEx1}, for some $r>0$. Notice that the following relation holds true
\begin{equation}\label{Eq2AuxEx1}
\nabla^2\ln(f)=\frac{1}{f}\nabla^2f-\frac{1}{f^2}df\otimes df.
\end{equation}
From \eqref{EqAuxEx1} and \eqref{Eq2AuxEx1}, we get
\begin{eqnarray}\label{Eq3AuxEx1}
\nonumber\lambda g &=& Ric+\nabla^2\varphi -r\nabla^2\ln(f) - \frac{r}{f^2}df\otimes df\\
&=&Ric+\nabla^2\xi -\frac{1}{r}d\xi\otimes d\xi + \nabla^2\varphi,
\end{eqnarray}
where $\xi:=-r\ln(f)$, which completes the proof of our assertion. This turn out that the difference between this classes is in the non-homotheticity of $\nabla\varphi$.

Another interesting situation is when we allow $\lambda$ to be a smooth function on the manifold. In this case, we have the following example which may be very useful in further studies.
\begin{example}\label{ex2}
Let $(\Bbb{M}^n(\tau),g_{\circ})$ be the standard sphere $\Bbb{S}^n$ or the hyperbolic space $\Bbb{H}^n$ for $\tau=1$ or $\tau=-1$ respectively. We denote
$h_v$ a height function with respect to a fixed unit vector $v\in\Bbb{R}^{n+1}$.
Then for each real number $m\neq0$, the functions $\lambda=\tau(n-1) -\frac{\tau}{m}h_v^2- h_v$, $f=e^{-\frac{\tau}{m}h_v}$ and $\varphi=\frac{1}{2m}h_v^2$ satisfy equation \eqref{EQMthmIntr} on $(\Bbb{M}^n(\tau),g_{\circ})$, once
\begin{equation*}
d\varphi=\frac{h_v}{m}dh_v, \quad df=-\frac{\tau}{m}e^{-\frac{\tau}{m}h_v} dh_v \quad \mbox{and}\quad \nabla^2 h_v=-\tau h_v g_\circ,
\end{equation*}
we get
\begin{equation*}
\nabla^2\varphi=\frac{1}{m}dh_v\otimes dh_v-\frac{\tau}{m}h_v^2g_{\circ}\quad \mbox{and}\quad \frac{m}{f}\nabla^2f=\frac{1}{m}dh_v\otimes dh_v+ h_vg_{\circ}.
\end{equation*}
On the other hand, $Ric=\tau(n-1)g_{\circ},$ so it is sufficient to choose $\lambda$ as at present example in order to obtain our desired statement.
\end{example}

Now, we can identify any $(0,2)$-tensor $T$ on $M$ with a  $(1,1)$-tensor by the equation
\begin{equation*}
g(T(Z), Y)=T(Z,Y)
\end{equation*}
for all $Y, Z \in \mathfrak{X}(M)$. Thus, we get
\begin{equation*}
\dv(\varphi T)=\varphi \dv T+ T(\nabla\varphi,\cdot)\quad \mbox{and}\quad \nabla(\varphi T)=\varphi\nabla T+d\varphi\otimes T
\end{equation*}
for all $\varphi\in C^\infty(M).$ In particular, we have $\dv(\varphi g)=d\varphi$. Moreover, the following general facts are well known in the literature
\begin{equation*}
 \dv\nabla^2\varphi = Ric(\nabla\varphi,\cdot)+ d\Delta\varphi \quad \mbox{and}\quad \frac{1}{2}d|\nabla\varphi|^2 = \nabla^2\varphi(\nabla\varphi,\cdot).
\end{equation*}

These identities will be used in what follows without further comments.

\begin{proposition}\label{PP5}
Let $(B^{n},g)$ be a Riemannian manifold with two smooth functions $f>0$ and $\varphi$ satisfying
\begin{equation}\label{EQMthm}
Ric+\nabla^{2}\varphi=\lambda g+\frac{m}{f}\nabla^{2}f \quad \mbox{and} \quad 2\lambda\varphi-|\nabla\varphi|^{2}+\Delta\varphi+\frac{m}{f}\nabla\varphi(f)=c
\end{equation}
for some constants $m,c,\lambda\in\Bbb{R}$, with $m\neq 0$. Then $f$ and $\varphi$ satisfy
\begin{equation}\label{CMthm}
\lambda f^{2}+f\Delta f +(m-1)|\nabla f|^{2} -f\nabla\varphi(f)=\mu
\end{equation}
for a constant $\mu\in\Bbb{R}.$
\end{proposition}
\begin{proof} From \eqref{EQMthm} we obtain $S=n\lambda+\frac{m}{f}\Delta f-\Delta\varphi$, where $S$ is the scalar curvature of $B$. Thus,
\begin{equation} \label{Bianch1}
dS=-\frac{m}{f^{2}}\Delta f df+\frac{m}{f}d(\Delta f)-d(\Delta\varphi).
\end{equation}
Let us now use the second contracted Bianch identity, namely
\begin{equation}\label{Biancc}
0=-\frac{1}{2}dS+\dv Ric.
\end{equation} We compute
\begin{eqnarray*}
\dv Ric&=&m\dv\Big(\frac{1}{f}\nabla^{2}f\Big)-\dv(\nabla^{2}\varphi)\\
&=&m\Big(\frac{1}{f}\dv(\nabla^{2}f)-\frac{1}{f^{2}}(\nabla^{2}f)(\nabla f,\cdot)\Big)-\dv(\nabla^{2}\varphi)\\
&=&\dfrac{m}{f}Ric(\nabla f,\cdot)+\frac{m}{f}d(\Delta f)-\frac{m}{2f^{2}}d(|\nabla f|^{2})-Ric(\nabla\varphi,\cdot)-d(\Delta\varphi).
\end{eqnarray*}
From \eqref{EQMthm} we have
\begin{eqnarray*}
Ric(\nabla f,\cdot)&=&\lambda df+\dfrac{m}{2f}d(|\nabla f|^{2})-(\nabla^{2}\varphi)(\nabla f,\cdot),\\
Ric(\nabla \varphi,\cdot)&=&\lambda d\varphi+\dfrac{m}{f}(\nabla^{2}f)(\nabla\varphi,\cdot)-\frac{1}{2}d(|\nabla\varphi|^{2}).
\end{eqnarray*}
This way
\begin{eqnarray*}
\dv Ric&=&\dfrac{m}{f}\Big(\lambda df+\dfrac{m}{2f}d(|\nabla f|^{2})-(\nabla^{2}\varphi)(\nabla f,\cdot)\Big)+\frac{m}{f}d(\Delta f)-\frac{m}{2f^{2}}d(|\nabla f|^{2})\\
&&-\Big(\lambda d\varphi+\dfrac{m}{f}(\nabla^{2}f)(\nabla\varphi,\cdot)-\frac{1}{2}d(|\nabla\varphi|^{2})\Big)-d(\Delta\varphi)\\
&=&\dfrac{m}{f}\lambda df+\dfrac{m^{2}}{2f^{2}}d(|\nabla f|^{2})-\dfrac{m}{f}(\nabla^{2}\varphi)(\nabla f,\cdot)+\frac{m}{f}d(\Delta f)-\frac{m}{2f^{2}}d(|\nabla f|^{2})\\
&&-\lambda d\varphi-\dfrac{m}{f}(\nabla^{2}f)(\nabla\varphi,\cdot)+\frac{1}{2}d(|\nabla\varphi|^{2})-d(\Delta\varphi)\\
&=&\dfrac{m}{f}\lambda df+\dfrac{m(m-1)}{2f^{2}}d(|\nabla f|^{2})+\frac{m}{f}d(\Delta f) -\lambda d\varphi+\frac{1}{2}d(|\nabla\varphi|^{2})-d(\Delta\varphi)\\
&&-\dfrac{m}{f}\big[(\nabla^{2}\varphi)(\nabla f,\cdot)+(\nabla^{2}f)(\nabla\varphi,\cdot)\big].
\end{eqnarray*}
Since $d(\nabla\varphi(f))=(\nabla^{2}\varphi)(\nabla f,\cdot)+(\nabla^{2}f)(\nabla\varphi,\cdot)$, then
\begin{eqnarray}\label{Bianch2}
\nonumber\dv Ric&=&\dfrac{m}{f}\lambda df+\dfrac{m(m-1)}{2f^{2}}d(|\nabla f|^{2})+\frac{m}{f}d(\Delta f)-\lambda d\varphi+\frac{1}{2}d(|\nabla\varphi|^{2})-d(\Delta\varphi)\\
&& -\dfrac{m}{f}d(\nabla\varphi(f)).
\end{eqnarray}
Plugging the equations \eqref{Bianch1} and \eqref{Bianch2} into equation \eqref{Biancc} we have
\begin{eqnarray*}
0&=&\frac{m}{2f^{2}}\Delta fdf-\frac{m}{2f}d(\Delta f)-\frac{1}{2}d(\Delta\varphi)+\dfrac{m}{f}\lambda df+\dfrac{m(m-1)}{2f^{2}}d(|\nabla f|^{2})+\frac{m}{f}d(\Delta f)\\
&&-\lambda d\varphi+\frac{1}{2}d(|\nabla\varphi|^{2})-\dfrac{m}{f}d(\nabla\varphi(f)).
\end{eqnarray*}
Multiplying the previous equation by $\frac{2f^{2}}{m}$ we get
\begin{eqnarray*}
0&=& \Delta fdf - fd(\Delta f) - \frac{f^{2}}{m}d(\Delta\varphi) + 2f\lambda df+(m-1)d(|\nabla f|^{2})+2fd(\Delta f)\\
&&-\frac{2f^{2}}{m}\lambda d\varphi+\frac{f^{2}}{m}d(|\nabla\varphi|^{2})-2fd(\nabla\varphi(f)).
\end{eqnarray*}
Simplifying and regrouping the terms, we obtain
\begin{equation}\label{01}
0=d\big(f\Delta f+\lambda f^{2}+(m-1)|\nabla f|^{2}\big)-\frac{f^{2}}{m}d\big(\Delta\varphi+2\lambda\varphi-|\nabla\varphi|^{2}\big)-2fd(\nabla\varphi(f)).
\end{equation}
But, by hypothesis $2\lambda\varphi-|\nabla\varphi|^{2}+\Delta\varphi+\frac{m}{f}\nabla\varphi(f)=c$. Whence
\begin{equation}\label{02}
-\frac{f^{2}}{m}d(\Delta\varphi+2\lambda\varphi-|\nabla\varphi|^{2})-fd(\nabla\varphi(f))= -\nabla\varphi(f)df.
\end{equation}
Consequently, equations \eqref{01} and \eqref{02} infer
\begin{eqnarray*}
d\big(f\Delta f+\lambda f^{2}+(m-1)|\nabla f|^{2}-f\nabla\varphi(f)\big)=0,
\end{eqnarray*}
which is sufficient to complete the proof.
\end{proof}

\section{Proof of the main results}

\subsection{Proof of Theorem \ref{mainthm}}
\begin{proof}
If $M=B^n\times_{f}F^m$,  $m>1$, is a gradient Ricci soliton with $Ric+\nabla^{2}\tilde{\varphi}=\lambda g$, then  Proposition \ref{PP2} implies  $^F\!Ric=\mu g_{F}$ where
\begin{equation}\label{Ric12}
\mu=\lambda f^{2}+f\Delta f+(m-1)|\nabla f|^{2}-f\nabla\varphi(f).
\end{equation}
By Proposition \ref{PP5} $\mu$ is constant, where the equations in \eqref{EQMthm} are guaranteed from equations \eqref{EqHteq} and \eqref{EqNecessary}. Let $p,q\in B^n$ be the points where $f$ attains its maximum and minimum in $B^n$. Then
\begin{equation*}
\nabla f(p)=0=\nabla f(q)\quad \mbox{and}\quad \Delta f(p)\leq 0\leq \Delta f(q).
\end{equation*}
Since $f>0$ and $\lambda\leq0$ we have $-\lambda f(p)^{2}\geq-\lambda f(q)^{2}$ and combining this with \eqref{Ric12} we get
\begin{equation} \label{Ric13}
0\geq f(p)\Delta f(p)=\mu-\lambda f(p)^{2}\geq \mu-\lambda f(q)^{2}=f(q)\Delta f(q)\geq 0.
\end{equation}
This last equation now implies
\begin{eqnarray*}
\mu-\lambda f(p)^{2}=\mu-\lambda f(q)^{2}=0.
\end{eqnarray*}
Thus, $\lambda<0$ infers that  $f(p)=f(q)$, i.e., $f$ is constant. For $\lambda=0$ we have that $\mu=0$  and equation \eqref{Ric12} reduces to
\begin{equation*}
Lf=\Delta f-\nabla\varphi(f)=\frac{1}{f}(1-m)|\nabla f|^{2}\leq 0,
\end{equation*}
where $L:=\Delta-\nabla\varphi$. Therefore, by the strong maximum principle $f$ is constant. In either case $M$ is a Riemannian product.
\end{proof}

\begin{remark}\label{RemNonTriv}
The warping function $f$ does not reach a minimum if $\mu\leq 0$ and $\lambda>0$. Indeed,  under the latter hypothesis \eqref{Ric12} implies $Lf<0$. Now, if $f$ reaches a minimum, we have by the strong maximum principle that $f$ is constant, which contradicts \eqref{Ric12}.
\end{remark}

\subsection{Proof of Theorem \ref{thmCpt}}
\begin{proof}
Assume that $B^n\times_{f}F^m$, $m>1$, is a gradient Ricci soliton with $Ric+\nabla^{2}\tilde{\varphi}=\lambda g$. As in the proof of Theorem \ref{mainthm} we have $^F\!Ric=\mu g_{F}$, where the constant $\mu$ is given by \eqref{Ric12} or equivalently
\begin{equation*}
\mu=\lambda f^{2}+fLf+(m-1)|\nabla f|^{2}.
\end{equation*}
By integration
\begin{equation*}
\mu\,\mathrm{vol}_\varphi(B^n)=\lambda\int_{B^n}f^2e^{-\varphi}\mathrm{dB} +(m-2)\int_{B^n}|\nabla f|^{2}e^{-\varphi}\mathrm{dB}.
\end{equation*}
Since $\lambda>0$ and $m>1$ we conclude that $\mu>0$ and so  $F^m$ is compact by the Bonnet-Myers Theorem. Thus, $B^n\times F^m$ is a compact manifold.
\end{proof}

\begin{remark}
Notice that Theorem \ref{thmCpt} has the following alternative proof. Recall first that if both $B^n$ and $F^m$ are complete Riemannian manifolds then $M=B^n\times_{f}F^m$ is complete for every warping function $f$. Since $\nabla\tilde\varphi\in\mathfrak{L}(B)$ and $B^n$ is compact we must have that $|\nabla\tilde\varphi|$ is bounded on the shrinking Ricci soliton $(M,g,\nabla\tilde\varphi,\lambda)$. Therefore, we can apply Theorem 1 in \cite{FLGR} to affirm that $M$ is compact manifold.
\end{remark}

\begin{remark}\label{remark3}
It is known that non-trivial compact Ricci solitons only exist in dimensions $k\geq4$ (see \cite{ham3,Ivey} or \cite{cao}). In particular, any shrinking gradient Ricci soliton warped product with compact base of dimension one and fiber with dimension two must be trivial, since by Theorem \ref{thmCpt} it is compact.
\end{remark}

It now arises the following  natural question: {\it Is it possible to construct a gradient Ricci soliton warped product with compact base and non-constant warping function?} Corollaries \ref{corCpt} and \ref{corCpt2} give a partial answer to this question.

\begin{corollary}\label{corCpt}
It is not possible to construct a gradient Ricci soliton warped product with compact base and non-constant warping function, so that its fiber is a Riemannian manifold of dimension at least two and of non-positive scalar curvature.
\end{corollary}

\begin{proof}
Suppose that $M^k$ is a gradient Ricci soliton warped product with compact base and non-constant warping function having as fiber a Riemannian manifold of dimension at least two. By Theorem \ref{mainthm}, $M^k$ must be shrinking Ricci soliton. Hence, as in the proof of Theorem \ref{thmCpt} its fiber must be an Einstein manifold with positive constant scalar curvature.
\end{proof}

\subsection{Proof of Theorem \ref{CRSPW}}
\begin{proof}
By the hypotheses on $f$ and $\varphi$ we can conclude from Proposition \ref{PP5} that any $\mu$ given by \eqref{CMthmIntr} is constant. Now, taking an Einstein manifold $(F^m,g_F)$ with Ricci tensor $^F\!Ric=\mu g_F$, we can consider the warped product $(B^n\times_fF^m, g)$ with $g=\pi^*g_B+(f\circ\pi)^2\sigma^*g_F$. Notice that this manifold has a structure of Ricci soliton. In fact, we observe that it follows from $H^\varphi(Y,Z)=\nabla^2\tilde\varphi(Y,Z)$, $H^f(Y,Z)=\nabla^2\tilde f(Y,Z)$, part $(i)$ of Lemma \ref{oneil2} and the hypothesis \eqref{EQMthmIntr} that the fundamental equation
\begin{equation*}
Ric+\nabla^2\tilde\varphi=\lambda g
\end{equation*}
is satisfied for all $Y,Z\in\mathfrak{L}(B)$. When $Y\in\mathfrak{L}(B)$ and $V\in\mathfrak{L}(F)$ we use  $\nabla\tilde\varphi\in\mathfrak{L}(B)$ and part $(i)$ of Lemma \ref{oneil1} to verify that $\nabla^2\tilde\varphi(Y,V)=g(D_Y\nabla\tilde\varphi,V)=0$. So, by part $(ii)$ of Lemma \ref{oneil2}, the fundamental equation is again satisfied.

Finally, for $V,W\in\mathfrak{L}(F)$ we have by definition of $\mu$ and part $(iii)$ of Lemma \ref{oneil2} that
\begin{eqnarray}
\nonumber Ric(V,W)&=&\mu g_F(V,W)-(f\Delta f+(m-1)|\nabla f|^2)g_F(V,W)\\
\nonumber &=&(\lambda f^2-f\nabla\varphi(f))g_F(V,W)\\
\label{eqAux2thm2}&=& (\lambda -\frac{1}{f}\nabla\varphi(f))g(V,W)).
\end{eqnarray}
On the other hand, equation \eqref{eqAux1Thm2}  gives us
\begin{equation}\label{eqAux3thm2}
\nabla^{2}\tilde{\varphi}(V,W) = f\nabla\varphi(f)g_{F}(V,W)=\frac{1}{f}\nabla\varphi(f)g(V,W).
\end{equation}
Combining equations \eqref{eqAux2thm2} and \eqref{eqAux3thm2} we conclude that the fundamental equation is again satisfied, which completes the proof of Theorem \ref{CRSPW}.
\end{proof}

As an application we construct the following class of expanding Ricci solitons. Recall first that the Gaussian soliton is the Euclidean space $\Bbb{R}^n$ endowed with its standard metric $g_\circ$ and the potential function $\psi(x)=\frac{\lambda}{2}|x|^2$. Moreover, let us consider on $\Bbb{R}^n$ the metric $\bar{g}=\frac{1}{\rho^2}g_\circ$, where $\rho$ is a positive smooth function on $\Bbb{R}^n$. We need to find two smooth functions $f>0$ and $\varphi$ on $\Bbb{R}^n$ and a constant $\lambda$ satisfying a Ricci-Hessian type equation on $(\Bbb{R}^n,\bar{g})$, i.e.,
\begin{equation}\label{eq1}
Ric_{\bar{g}}+\bar\nabla^2\varphi=\lambda\bar{g}+\frac{m}{f}\bar\nabla^2f,
\end{equation}
where $m>0$ is an integer. Considering the theoretical facts we obtain
\begin{equation}\label{eq2}
Ric_{\overline{g}}=\frac{1}{\rho^2}\big\{(n-2)\rho\nabla^2\rho+(\rho\Delta\rho-(n-1)|\nabla\rho|^2)g_\circ\big\},
\end{equation}
where the two summands appearing in the second term of this equation are calculated in the metric $g_\circ$. Moreover, for every $h\in C^\infty(\mathbb{R}^n)$ the following are valid
\begin{eqnarray*}
(\bar\nabla^2h)_{ij}&=&h_{x_ix_j}+\frac{\rho_{x_j}}{\rho}h_{x_i}+\frac{\rho_{x_i}}{\rho}h_{x_j}\quad \mbox{for}\quad i\neq j,\\
(\bar\nabla^2h)_{ii}&=&h_{x_ix_i}+2\frac{\rho_{x_i}}{\rho}h_{x_i}-\sum_k\frac{\rho_{x_k}}{\rho}h_{x_k}\quad \mbox{for}\quad i=j.
\end{eqnarray*}
So we need to analyze equation \eqref{eq1} in two cases. For $i\neq j$, it rewrites as
\begin{equation}\label{eq5}
(n-2)\frac{\rho_{x_ix_j}}{\rho}+\varphi_{x_ix_j}+\frac{\rho_{x_j}}{\rho}\varphi_{x_i}+\frac{\rho_{x_i}}{\rho}\varphi_{x_j}
=\frac{m}{f}\Big(f_{x_ix_j}+\frac{\rho_{x_j}}{\rho}f_{x_i}+\frac{\rho_{x_i}}{\rho}f_{x_j}\Big)
\end{equation}
and for $i=j$,
\begin{eqnarray}\label{eq6}
\nonumber &&(n-2)\frac{\rho_{x_ix_i}}{\rho}+\frac{1}{\rho}\sum_k\rho_{x_kx_k}-\frac{(n-1)}{\rho^2}\sum_k\rho_{x_k}^2+\varphi_{x_ix_i}+2\frac{\rho_{x_i}}{\rho}\varphi_{x_i}\\
&=&\frac{1}{\rho}\sum_k\rho_{x_k}\varphi_{x_k}+\frac{\lambda}{\rho^2}+\frac{m}{f}\Big(f_{x_ix_i}
+2\frac{\rho_{x_i}}{\rho}f_{x_i}-\frac{1}{\rho}\sum_k\rho_{x_k}f_{x_k}\Big).
\end{eqnarray}
We assume that the functions have the following dependencies $\rho=\rho(x_n)$, $f=f(x_n)$ and $\varphi=\varphi(y)$,
where $x=(y,x_n)\in\Bbb{R}^n$ and $y=(x_1,\ldots,x_{n-1}).$ So taking $i=n$ on \eqref{eq5} we have for all $j\neq n$
\begin{equation*}
\rho_{x_n}\varphi_{x_j}=0.
\end{equation*}
Again from \eqref{eq5}, for $i\neq n$, we obtain for all $i\neq j\neq n$,
\begin{equation}\label{eq6.1}
\varphi_{x_ix_j}=0.
\end{equation}
As we are interested in obtaining nontrivial solutions for \eqref{eq1}, we should consider $\rho$ constant. Thus, from \eqref{eq6} we get
\begin{equation}\label{eq7}
\varphi_{x_ix_i}-\frac{m}{f}f_{x_ix_i}=\frac{\lambda}{\rho^2}.
\end{equation}
Whence, for $i\neq n$, we have
\begin{equation}\label{eq8}
\varphi_{x_ix_i}=\frac{\lambda}{\rho^2}.
\end{equation}
Hence, from \eqref{eq6.1} and \eqref{eq8}, $\varphi$ it is well determined by
\begin{equation}\label{eq8.1}
(D^2\varphi)_{ij}=\frac{\lambda}{\rho^2}\delta_{ij},
\end{equation}
where $D^2\varphi$ stands for the Hessian of $\varphi$ calculated in the metric $\delta_{ij}$ on $\Bbb{R}^{n-1}$. Moreover, taking $i=n$ on \eqref{eq7}, we obtain
\begin{equation}\label{eq9}
f_{x_nx_n} + \frac{\lambda}{m\rho^2}f=0.
\end{equation}
Therefore, by the theory of ODE's, it only remains for us to choose $\lambda<0$ in order to obtain
\begin{equation}\label{eq10}
f(x)=c_1e^{\frac{1}{\rho}\sqrt{\frac{-\lambda}{m}}x_n}+c_2e^{-\frac{1}{\rho}\sqrt{\frac{-\lambda}{m}}x_n}>0
\end{equation}
for every non-negative constants $c_1$ and $c_2$ which assure $f>0$.

The conclusion is that for every negative constant $\lambda$, the smooth functions $\varphi$ and $f$ respectively given by \eqref{eq8.1} and \eqref{eq10} satisfy a Ricci-Hessian type equation \eqref{eq1} on $(\Bbb{R}^n,\frac{1}{\rho^2}g_\circ)$, where $\rho>0$ is constant.

\begin{corollary}\label{CorCRSWPgeneral}
Let $\Bbb{R}^n$ be a Euclidean space with Euclidian metric and coordinates $x=(y,x_n)$, where $y=(x_1,\ldots,x_{n-1})$ and $n>1$. Consider a complete Riemannian manifold $(F^m,g_F)$ with Ricci tensor $^F\!Ric=\mu g_F$ and $m>1$. Then $(\Bbb{R}^n\times_fF^m,g, \nabla\tilde\varphi,\lambda)$ has a structure of expanding gradient Ricci soliton with $\mu\leq0$, where
\begin{equation}\label{eqAuxCor}
f(x)=c_1e^{\sqrt{\frac{-\lambda}{m}}x_n}+c_2e^{-\sqrt{\frac{-\lambda}{m}}x_n}>0 \quad \mbox{and} \quad \varphi(x)=\frac{\lambda}{2}|y|^2+\sum_{i=1}^{n-1}a_ix_i+b
\end{equation}
for any vector $a=(a_1,\ldots,a_{n-1})\in\Bbb{R}^{n-1}$, $b\in \Bbb{R}$, and every non-negative constants $c_1$ and $c_2$ which assure $f>0$.
\end{corollary}
\begin{proof}
Taking $\rho=1$ in \eqref{eq8.1} and \eqref{eq10}, we can consider the functions in \eqref{eqAuxCor}
satisfying equation \eqref{EQMthm} for $c=2\lambda b-|a|^2+\lambda(n-1)$. Consequently, from equation \eqref{CMthm} we compute
\begin{equation*}
\mu = \lambda f^2-\frac{\lambda}{m}f^2+(m-1)f_{x_n}^2 =(m-1)\Big(\frac{\lambda}{m}f^2+f_{x_n}^2\Big).
\end{equation*}
A straightforward calculation shows
\begin{equation*}
\mu=4\lambda c_1c_2\frac{(m-1)}{m}\leq0.
\end{equation*}
The conclusion of the corollary now immediately follows from Theorem \ref{CRSPW}.
\end{proof}

\section{Concluding remarks}

According to Corollary \ref{corCpt} the compactness of the base of a Ricci soliton warped product implies restrictions on its existence. For this reason, we establish the compactness criterion below, whose demonstration follows the same technique used by Fern\'andez-L\'opez and Garc\'ia-R\'io \cite{FLGR}.

\begin{proposition}\label{crit-Cpt}
Let $(B,g)$ be a complete Riemannian manifold satisfying
\begin{equation}\label{compacta}
Ric+\nabla^2\varphi-\frac{m}{f}\nabla^2f\geq cg
\end{equation}
for $m>0$, some smooth functions $\varphi,f$ on $B$ and some positive constant $c$. Then, $B$ is compact provided that $|\nabla\varphi|$ and $|\nabla(\ln f)|$ are both bounded on $(B,g)$. In particular, the fundamental group $\pi_1(B)$ is finite.
\end{proposition}
\begin{proof}
Let $p$ be a point in $B$ and consider any geodesic $\gamma:[0,+\infty)\rightarrow B$ emanating from $p$ and parameterized by arc length $s.$ From inequality \eqref{compacta} we have
\begin{eqnarray*}
Ric(\gamma',\gamma')&\geq&cg(\gamma',\gamma')-g(\nabla_{\gamma'}\nabla\varphi,\gamma')+\frac{m}{f}g(\nabla_{\gamma'}\nabla f,\gamma')\\
&=&c-\frac{d}{ds}g(\nabla\varphi,\gamma')+m\frac{d}{ds}g(\nabla(\ln f),\gamma') + \frac{m}{f^2}(\gamma'(f))^2.
\end{eqnarray*}
Thus, by integrating and using the Cauchy-Schwarz inequality, we get
\begin{eqnarray*}
\int_0^tRic(\gamma',\gamma')ds&\geq&ct+g(\nabla\varphi_p,\gamma'(0))-g(\nabla\varphi_{\gamma(t)},\gamma'(t)) -mg(\nabla(\ln f)_p,\gamma'(0))\\
& &-mg(\nabla(-\ln f)_{\gamma(t)},\gamma'(t))\\
&\geq&ct+g(\nabla\varphi_p,\gamma'(0))-|\nabla\varphi_{\gamma(t)}| - mg(\nabla(\ln f)_p,\gamma'(0))\\
&&-m|\nabla(\ln f)_{\gamma(t)}|.
\end{eqnarray*}
Since $|\nabla\varphi|$ and $|\nabla (\ln f)|$ are both bounded, we obtain
\begin{equation*}
\int_0^{+\infty}Ric(\gamma',\gamma')ds=+\infty.
\end{equation*}
So, the compactness of $(B,g)$ follows from the Ambrose's compactness criterion \cite{Ambrose}. Finally, observe that equation \eqref{compacta} will also hold true in the universal cover of $B$, which will imply the compactness of the latter, and thus $\pi_1(B)$ is finite.
\end{proof}

\begin{corollary}\label{corCpt2}
It is not possible to construct a gradient Ricci soliton warped product $M^k$ with compact base $B$ and non-constant warping function, so that its fiber is a Riemannian manifold of dimension at least two and $\pi_1(B)$ is not finite.
\end{corollary}

\begin{proof}
As in the proof of Corollary \ref{corCpt}, $M^k$ must be shrinking Ricci soliton warped product. Hence, by Propositions \ref{PP2} and \ref{crit-Cpt}, its base must be a Riemannian manifold with finite fundamental group.
\end{proof}

\textbf{Acknowledgements:} The authors would like to express their sincere thanks to D. Tsonev and E. Ribeiro Jr. for useful comments, discussions and constant encouragement. The authors also thank the referee for careful reading and useful comments which improved the paper. This work has been partially supported by CAPES-Brazil, FAPEAM-Brazil and CNPq-Brazil.

\end{document}